\documentclass[12pt]{amsart}

\usepackage{amsthm,amssymb,amsmath,enumerate,graphicx, tikz}
\usepackage{amscd}
\usepackage{setspace}
\usepackage{hyperref}
\newtheorem{theorem}{Theorem}[section]

\newtheorem*{theorem*}{Theorem}

\newtheorem{lemma}[theorem]{Lemma}
\newtheorem{observation}[theorem]{Observation}
\newtheorem{corollary}[theorem]{Corollary}
\newtheorem{claim}[theorem]{Claim}
\newtheorem{conjecture}[theorem]{Conjecture}
\theoremstyle{definition}
\newtheorem{definition}[theorem]{Definition}

\theoremstyle{remark}
\newtheorem{remark}[theorem]{Remark}
\newtheorem{example}[theorem]{Example}

\newcommand{\cp}{\mathcal{P}}

\newcommand{\cs}{\mathcal{S}}
\newcommand{\ca}{\mathcal{A}}
\newcommand{\cb}{\mathcal{B}}

\newcommand{\D}{\mathcal{D}}

\newcommand{\A}{\mathcal{A}}

\newcommand{\Q}{\mathcal{Q}}

\newcommand{\B}{\mathcal{B}}

\newcommand{\tp}{\tilde{\mathcal{P}}}

\title{Badges and rainbow matchings}

\author{Ron Aharoni}
\address{Department of Mathematics, Technion -- Israel Institute of Technology, Technion City, Haifa, Israel and Moscow Institute of Physics and Technology, Dolgoprudny, Russia.}
\email{ra@tx.technion.ac.il}

\author{Joseph Briggs}
\address{Department of Mathematics,
Technion -- Israel Institute of Technology, Haifa, Israel.}
\email{briggs@campus.technion.ac.il}

\author{Jinha Kim}
\address{Discrete Mathematics Group, Institute for Basic Science (IBS), Daejeon, South Korea.} \email{jinhakim@ibs.re.kr}
\thanks{Jinha Kim(\texttt{jinhakim@ibs.re.kr}) is the corresponding author.}

\author{Minki Kim}
\address{Discrete Mathematics Group, Institute for Basic Science (IBS), Daejeon, South Korea.} \email{minkikim@ibs.re.kr}

\date{\today}

\begin{document}

\begin{abstract}
Drisko proved that $2n-1$ matchings of size $n$ in a bipartite graph have a rainbow matching of size $n$. For general graphs it is conjectured that $2n$ matchings suffice for this purpose (and that $2n-1$ matchings suffice when $n$ is even). The known graphs showing sharpness of this conjecture 
for $n$ even are called {\em badges}.
We improve the previously best known bound from $3n-2$ to $3n-3$, using a new line of proof that involves analysis of the appearance of badges. 
We also prove a ``cooperative'' generalization: for $t>0$ and $n \geq 3$, any
$3n-4+t$ sets of edges, the union of every $t$ of which contains a matching of size $n$, have a rainbow matching of size $n$. 
\end{abstract}

\maketitle

\section{Introduction}
Given a collection of sets, $\cs=(S_1, \ldots,S_m)$, an $\cs$-{\em rainbow set} is 
the image of a partial choice function of $\cs$. So, it is a set $\{x_{i_j}\}$, where $1 \le i_1< \ldots <i_k\le m$ and $x_{i_j}\in S_{i_j} ~(j \le k)$.  We call the sets $S_i$ {\em colors} and we say that $x_{i_j}$ is {\em colored by $S_{i_j}$}  and that $x_{i_j}$  {\em represents $S_{i_j}$} in the rainbow set.

Given  numbers $m,n,k$, we write $(m,n)\to k$ if every $m$ matchings of size $n$ in any graph  have a rainbow matching of size $k$, and $(m,n)\to_{\B} k$  if the same is true in every bipartite graph. 
 Generalizing a result of Drisko \cite{drisko}, 
  the first author and E. Berger proved \cite{ab}:
  
  \begin{theorem}\label{drisko}
  $(2n-1,n)\to_\B n$. 
  \end{theorem}  
  
  In \cite{abchs} it was conjectured that 
  almost the same is true in all graphs: 
     
  \begin{conjecture}\label{2nconj}
  $(2n, n) \to n$. If $n$ is odd then $(2n-1, n) \to n$.
  \end{conjecture}

If true, this is reminiscent of the relationship between K\"onig's theorem, stating that $\chi_e \le \Delta$ in bipartite graphs, and Vizing's theorem,  $\chi_e \le \Delta+1$ in general graphs, where $\chi_e$ is the edge chromatic number, namely the minimal number of matchings covering the edge set of the graph. The conjecture says that there is a price of just $1$ for passing from bipartite graphs  to general graphs.   

The conjecture is implicit already in  \cite{barat} - the examples showing sharpness appear there.  Another conjecture of the first author and E. Berger, stated  in the bipartite case in \cite{ab}, but presently no counterexample is known also in general graphs, is:
     \begin{conjecture}\label{ab}
     $(n,n)\to n-1$.
          \end{conjecture}

     If true, this conjecture implies a famous conjecture of Ryser--Brualdi--Stein \cite{br}. A {\em partial transversal} in a Latin square is a set of distinct entries, lying in different rows and different columns. 
     \begin{conjecture}\cite{br}
          Every  $n \times n$ Latin square contains a Latin transversal of size $n-1$.
    \end{conjecture}

     It also implies the first part, $(2n,n)\to n$, of Conjecture \ref{2nconj}. To see this, assume that  $M_1, \ldots ,M_{2n}$ are matchings of size $n$ in any graph. Let $V_1, V_2$ be two disjoint copies of the vertex set  of the graph, and for $j=1,2$~and ~$i \le 2n$  let $M^j_i$ be a copy  of $M_i$ on $V_j$. Let $N_i=M^1_i \cup M^2_i~~(i\le 2n)$. If Conjecture \ref{ab} is true, then the system $(N_1, \ldots ,N_{2n})$  has a rainbow matching $N$ of size $2n-1$. By the pigeonhole principle, $n$ of the edges of $N$ belong to the matchings $M_i^j$ for the same $j$, proving Conjecture \ref{2nconj}. If indeed this is the reason (in some non-rigorous sense) for the validity of Conjecture \ref{2nconj}, this may explain the difficulty of the latter, since as mentioned above Conjecture \ref{ab} belongs to a family of notoriously hard problems.

     The previously best result on Conjecture \ref{2nconj} is:
\begin{theorem}\label{3nminus2}
     \cite{abchs}  $(3n-2, n)\to n$.
     \end{theorem}
     
In \cite{lee}  an alternative, topological proof was given for this result.
Theorem \ref{drisko}  has been given a few proofs. Three quite distinct topological proofs were given in  \cite{ahj, abkz, lee}. There is also more than one combinatorial proof, but all known proofs are similar in spirit. All use alternating paths. 

\begin{definition}
Given a matching $F$ in a graph, a simple path is called $F$-{\em alternating} if every other edge in it belongs to $F$. It is called {\em augmenting} if it starts and ends at a vertex not belonging to any edge in $F$.
 Similarly, a  cycle of odd length is called {\em  $F$-alternating} if every other edge in it belongs to $F$, apart from two adjacent edges, that do  not belong to $F$.
\end{definition}

The origin of the name ``augmenting'' is that taking the symmetric difference of $F$ and the edge-set of the path yields a matching larger than $F$. 

The combinatorial proofs of Theorem \ref{drisko} are based on  the following:     
     \begin{theorem}\cite{akz}\label{thm: akz}
     If $F$ is a matching of size $k$ in a bipartite graph and $\A=(A_1, \ldots ,A_{k+1})$ is a family of augmenting $F$-alternating paths then there exists an $\A$-rainbow augmenting $F$-alternating path. 
     \end{theorem}
     
     For a given matching $F$, we write ``$F$-AAP''   for ``augmenting $F$-alternating path''. Given a family $\cp$ of sets of edges, an $F$-alternating path is called {\em $\cp$-rainbow} if its non-$F$ edges form a $\cp$-rainbow set. 

          To deduce Theorem \ref{drisko}, let $H$ be a family of $2n-1$ matchings of size $n$. Let $F$ be a maximal size rainbow matching, and assume for contradiction that $k:=|F|<n$. Then more than $k$ matchings are not represented in $F$. Since each of them is larger than $F$, each of them contains an $F$-AAP after taking the union with $F$. Taking the symmetric difference of $F$ and the rainbow alternating path provided by Theorem \ref{thm: akz} yields a  rainbow matching larger than $F$, a contradiction.

     The result $(3n-2, n)\to n$ follows in a similar way, using the following theorem:
     \begin{theorem}\cite{abchs}\label{t.abchs}
     If $F$ is a matching of size $k$ in any graph and $\A=(A_1, \ldots ,A_{2k+1})$ is a family of $F$-AAPs, then there exists a rainbow $F$-AAP. 
     \end{theorem}

     This is sharp: in the next section we shall introduce ``origamistrips'', that are families of $2|F|$ $F$-AAPs without a rainbow AAP. This means the strategy that works in the bipartite case cannot bring us close to 
     Conjecture \ref{2nconj}. But some additional effort 
    can take us one step further, which is the main result of this paper:

\begin{theorem}\label{thm: main}
$(3n-3,n)\to n$ for any $n \geq 3$.
\end{theorem}

The proof uses a characterization of the 
examples showing sharpness of Theorem \ref{t.abchs}, namely of those  families of $2k$ AAPs not having a rainbow AAP.  
     
    In the second part of the paper we  prove a ``cooperative'' generalization of the theorem. This means that we are not given matchings, but sets of edges, the union of every $t$ of which ($t$ being a parameter of the result) contains a matching of size $n$. The conclusion is, again, the existence of a rainbow matching of size $n$.

    \subsection{Paths terminology}
    We shall use paths extensively. The default assumption is that paths are undirected. 
Throughout the paper we shall tacitly identify a path with its edge set. So, a path is even if it has an even number of edges.    
    If $P = \{v_1v_2, v_2v_3, \ldots, v_{m-1}v_m\}$, $v_1$ and $v_m$ are called the {\em endpoints} of $P$ and $v_2,\ldots,v_{m-1}$ are called {\em interior vertices} of $P$.
    We sometimes write $P = v_1 v_2 \ldots v_m$.
    Note that $v_1 v_2 \ldots v_m$ and $v_m v_{m-1} \ldots v_1$ are the same path.
    
    For  paths $S = s_1 s_2 \ldots s_p$ and $T= t_1 t_2 \ldots t_q$ (taken in this case in a directed sense)  let $ST$ be the walk $s_1 s_2 \ldots s_pt_1 t_2 \ldots t_q$. In particular, if $x$ is a vertex then $Sx=s_1 s_2 \ldots s_px$ (these will be used below only when the resulting walks are paths).

\section{Badges}
\begin{definition}
An $m$-{\em origamistrip} $OS$ is a graph whose vertex set is
$\{u_1,\ldots u_m, v_1, \ldots ,v_m,x,y\}$ (all $u_i,v_i, x,y$
being distinct) and whose edge set is the union of three disjoint matchings: 
 $M=M(OS)=\{e_1=u_1v_1, \ldots,e_m=u_mv_m\}$; $A(OS)=\{xu_1,v_1u_2,v_2 u_3,\ldots 
 , v_{m-1}u_m, v_my\}$; and $B(OS)=\{xv_1,u_1v_2,\ldots v_{m-1}u_m ,u_m y\}$.
\end{definition}

See Figure~\ref{fig:origamistrips}.
\begin{figure}[ht]
    \centering
    \includegraphics[scale=0.8]{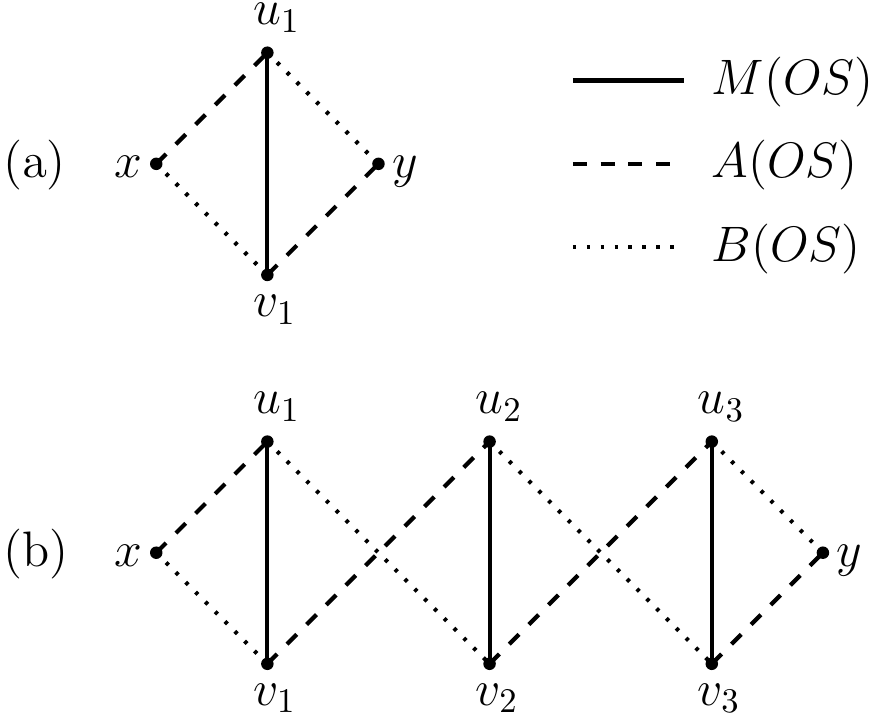}
    \caption{ (a) a $1$-origamistrip, and (b) a $3$-origamistrip.}
    \label{fig:origamistrips}
\end{figure}

There is no real difference between $A(OS)$ and $B(OS)$, the differentiation is  merely for notational convenience.
Then  $P^A(OS):=A(OS)\cup M$ and $P^B(OS):=B(OS)\cup M$ are $M$-AAPs.  

The vertices $x,y$ are called the {\em endpoints} of $OS$, and $u_1,v_1, \dots, u_m, v_m$ are called the {\em interior vertices}. The matching $M$ is called the {\em skeleton} of $OS$.

  We will want to refer to the $M$-alternating paths reaching vertices of the origamistrip, to be named below:   
\begin{observation}\label{qaqb}
 For every interior vertex $v$ of $OS$ there exist an even $M$-alternating path $Q_v^A(OS) \subseteq P^A(OS)$ from an endpoint of $OS$ to $v$ and 
 an even $M$-alternating path $Q_v^B(OS) \subseteq P^B(OS)$ from the other endpoint of $OS$ to $v$. If $x$ is an endpoint of $OS$, write $Q_v^x(OS)$, or simply $Q_v^x$, for whichever of $\{Q_v^A(OS), Q_v^B(OS)\}$ begins at $x$.
 \end{observation}

$m$-origamistrips can be used to show the sharpness of Conjecture \ref{2nconj}.

 \begin{observation}\label{sharp}
Let $n$ be even, and let $OS$ be an $(n-1)$-origamistrip with skeleton $M$. Taking $n-1$ copies of $A(OS)$, $n-1$ copies of $B(OS)$ and the matching $M \cup \{xy\}$ provides an example of $2n-1$ matchings of size $n$, having no rainbow matching of size $n$. 
 \end{observation}
 
  For $n$ even this is one of few examples for sharpness we know. For $n$ odd, the example does not work: it has a rainbow matching of size $n$, consisting of the edge $xy$ and pairs $u_{2i}v_{2i-1}, ~u_{2i-1}v_{2i}$.

The following explains why Observation \ref{sharp} is true:
\begin{observation}\label{noraap}
 Let $OS$ be an $m$-origamistrip with skeleton $M$, and let $\Q$ be the collection of paths containing each of $P^A(OS)$ and $P^B(OS)$, each repeated $m$ times. Then there is no $\Q$-rainbow $M$-AAP. 
 \end{observation}

We now introduce badges, which are a systematic way of combining multiple origamistrips.

\begin{definition}
A {\em badge}  $\cb$ is collection of paths, obtained from a multigraph $H$ and an integer weighting $w$ on its edge set, in the following way. For each edge $e=xy$ of $H$ let $OS(e)$ be a $w(e)$-origamistrip  with endpoints $x,y$, where all 
$OS(e)$s have disjoint interior vertex sets.
The badge $\cb$ is then the collection of all paths $P^A(OS(e)), P^B(OS(e))$, each repeated $w(e)$ times. Let $M(\cb)= \bigcup_{e \in E(H)}M(OS(e))$ and $w(\cb)=\sum_{e \in E(H)}w(e)$.
We also call  $\cb$ is a $k$-\emph{badge}, where $k=w(\cb)$, and we call  $M(\cb)$ its  \emph{skeleton}.
\end{definition}
Thus $\cb$ consists of $2w(\cb)$ many $M(\cb)$-AAPs. For $v \in \bigcup M(\cb)$ let $OS_v$ be the origamistrip in $\cb$ that contains $v$.

See Figure~\ref{fig:badge} for an example of an $8$-badge.
Here, $\mathcal{B}$ is the badge obtained from $H$ with integer weighting $w$, where $w(e_1) = 3$, $w(e_2) = w(e_4) = 2$, and $w(e_3) = 1$.
For each $i \in \{1,2,3,4\}$, $\mathcal{B}$ contains two paths $P^A(OS(e_i))$ and $P^B(OS(e_i))$ in $OS(e_i)$, each repeated $w(e_i)$ times.
Let us describe two kinds of paths of $\mathcal{B}$:
\begin{itemize}
    \item The union of the dotted edges and $M(OS(e_1))$ is $P^A(OS(e_1))$, repeated $3$ times.
    \item The union of the dashed edges and $M(OS(e_4))$ is $P^B(OS(e_4))$, repeated $2$ times.
\end{itemize}
\begin{figure}[ht]
    \centering
    \includegraphics[scale=1]{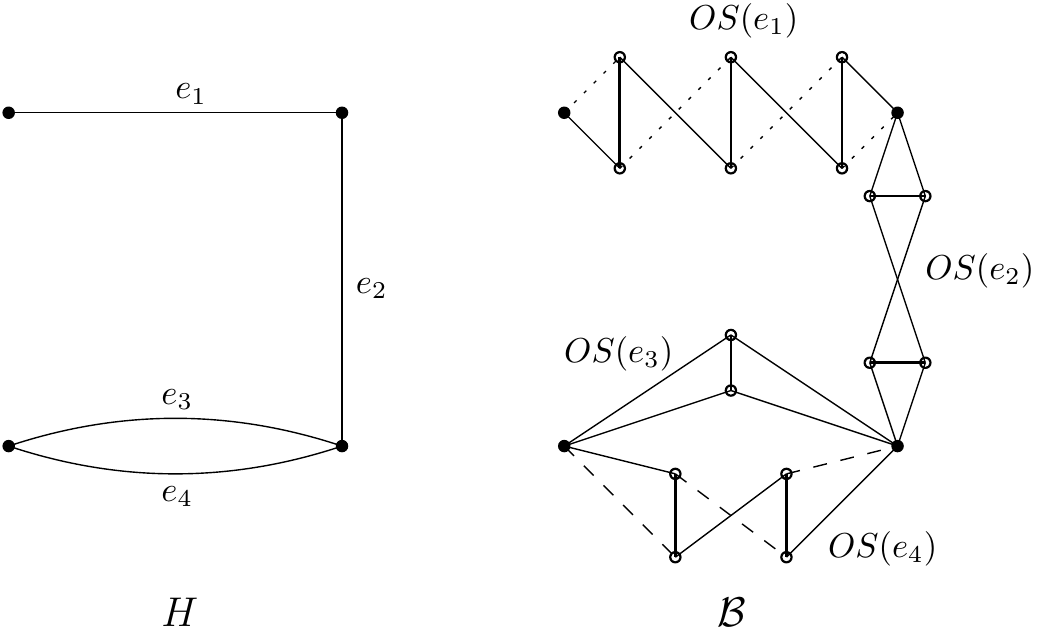}
    \caption{$\mathcal{B}$ is the badge obtained from $H$ with   weighting $w$, where  $w(e_1) = 3$, $w(e_2) = w(e_4) = 2$, and $w(e_3) = 1$.}
    \label{fig:badge}
\end{figure}

In the next section, on cooperative versions of the main theorem, we shall also use something more general:
\begin{definition}[generalized badges]\label{genbadge} 
Suppose we modify a badge $\B$ by allowing the addition and the deletion of edges from $M(\cb)$ to each $P^A(OS(e))$ and to each $P^B(OS(e))$. The resulting construction is then called a {\em generalized badge}.  
\end{definition}

By Observation \ref{noraap} there is no rainbow $M(\cb)$-AAP in any origamistrip, hence in all of $\cb$. We next show that $\cb$ is edge-maximal with respect to this property.

\begin{lemma}\label{lem: add an edge}
Let $\cp$ be a badge with skeleton $F$. 
 Let
$\cp^+$ be obtained from $\cp$ by replacing one path $P \in \cp$  by $P^+=P \cup \{xy\}$, where $xy \not \in P$ and $xy \not \in F$. Then  $\cp^+$
has a rainbow $F$-AAP.
\end{lemma}
\begin{proof}
If $x, y \notin \bigcup F$, then $xy$ is by itself a rainbow $F$-AAP. 
So, without loss of generality we may assume that $x \in \bigcup F$. 
Also without loss of generality, $P = P^A(OS_x)$.
If $y \notin \bigcup F$, then the $F$-AAP path $Q^B_x(OS_x) \cup \{xy\}$ (see Observation \ref{qaqb} for the definition of $Q^B_x(OS_x)$) is $\cp^+$-rainbow. The reason is that there are  enough paths in $B(OS_x)$ to color the edges of $Q^B_x(OS_x)$, while $P^+$ can color $xy$.  

So, we may assume that $y \in \bigcup F$ as well. 

\smallskip

{\bf Case I.} $OS_x \neq OS_y$. 
Let $z$ be an endpoint of $OS_y$ which is not an endpoint of $Q^B_x(OS_x)$. 
Then $Q^B_x(OS_x) \cup \{xy\} \cup Q^z_y(OS_y)$ is a $\cp^+$-rainbow $F$-AAP. 

{\bf Case II.}  $OS_x = OS_y$.
Then a path of the form $Q_x^p(OS_x) \cup \{xy\} \cup Q_y^q(OS_y)$, where each of $p,q$
is either $A$ or $B$, is a $\cp^+$-rainbow $F$-AAP. 
\end{proof}

\begin{corollary}
\label{rmk:extraedge}
If $\cp$ is a badge with skeleton $F$ and $\cp^+=\cp \cup \{\{xy\}\}$ for some $xy \notin F$, then $\cp^+$ has a rainbow $F$-AAP.
\end{corollary}
To see this,  add $xy$ to any path to which it does not belong, and use the lemma.

\section{Badges and alternating paths}
The main step towards the proof of Theorem \ref{thm: main} is showing that badges are  critical obstructions for the existence of rainbow AAPs.
 
\begin{theorem}\label{thm:rb.a.a.path}
Let $F$ be a matching of size $k$ in a  graph $G$, and let $\cp$ be a family of $F$-AAPs.
Suppose $\cp$ has no rainbow $F$-AAP.
Then,
\begin{enumerate}
    \item $|\cp| \leq 2k$, and 
    \item if $|\cp| = 2k$ then $\cp$ is a $k$-badge. 
\end{enumerate}
\end{theorem}

Part (1) is Theorem \ref{t.abchs}, which as remarked above was proved in \cite{abchs}. The proof  (occupying most of this section) differs from the proof given there - it is done inductively, together with (2). 

\begin{proof}[Proof of Theorem~\ref{thm:rb.a.a.path}]
Note, first, that  (2) implies (1).   Assume (2) is true for a given $k$. If (1) fails,  namely $|\cp| > 2k$,  let $\tp$ be a subset of $ \cp$ of size $2k+1$, let $P\in\tp$ and let  $\cb=\tp \setminus \{P\}$. By (2), $\cb$ is a $k$-badge, and by  Corollary~\ref{rmk:extraedge} $\cp \supseteq \tp = \cb\cup \{P\}$ has a rainbow $F$-AAP. 

So, it suffices to prove (2). We do this  by induction on $k$. The base case $k=0$ is vacuously true. For the inductive step,  assume (2) is true for $k-1$. 
Assume, for negation, that  $|\cp|=2k$  and   there is no $\cp$-rainbow $F$-AAP.
Let $U= \bigcup F$.

An outline of the proof is as follows.
\begin{enumerate}[(i)]
	\item We find some triangle $C = \{va, ab, bv\}$, where $ab \in F$ and $v \not \in U=\bigcup F$.

	\item Contracting $C$ to a vertex $w$ and applying the induction hypothesis results in a badge $\cb'$, containing an origamistrip $OS'$,  having $w$ as one of its endpoints. 

	\item Open up $OS'$ to an origamistrip $OS$ containing $ab$, such that $ab$ is the first $F$-edge on both $P^A(OS)$ and $P^B(OS)$. 

	\item Remove $OS$, namely its $m+1$ $F$-edges and $2m+2$ $F$-AAPs. This leaves the matching $F \setminus M(OS)$ of size $k-m-1$, along with $2k-2m-2$ $(F\setminus M(OS))$-AAPs, with no rainbow $F \setminus M(OS)$-AAP. By the induction hypothesis, these form a $(k-m-1)$-badge $\cb$ with skeleton $F \setminus M(OS)$.

	\item Add back the origamistrip $OS$ to $\cb$. Since the skeleton $M(OS)$ is disjoint from $F \setminus M(OS)$,  $\cp$ is  a badge,  proving (2).
\end{enumerate}

We now proceed with the proof in detail.

\begin{claim}\label{claim: odd cycle}
There exists an odd $\cp$-rainbow  $F$-alternating cycle, containing a vertex that does not belong to $U$.
\end{claim}
\begin{proof}
Choose  $Q\in \cp$, and let $va$ be its first edge, where $v \not \in U$ and $a \in U$. 
Let $b$ be the vertex matched to $a$ in $F$. 
If the pair $vb$ lies on a path belonging to $\cp$, then the triangle $vab$ is an odd rainbow cycle as desired. 
So, we may assume that $vb$ does not lie on a $\cp$ path. Let $F'=F\setminus \{ab\}$. 

Remove the vertex $a$.
Replace every $P \in \cp\setminus \{Q\}$ by an $F'$-AAP $P'$, as follows. 
\begin{itemize}
    \item If $P$ does not contain  $ab$, let $P' = P$. 
    \item If $P$ is a path of the form $P=RabS$, let $P'=bS$.
\end{itemize}

Let $\cp'=\{P' \mid P \in \cp \setminus \{Q\}\}$. 
The matching $F'$ is of size $k-1$, and $|\cp'|=2k-1 > 2(k-1)$. 
By the induction hypothesis for (1), there exists a $\cp'$-rainbow $F'$-AAP $R$. 
$R$ contains $b$, otherwise it is  a $\cp$-rainbow $F$-AAP, contrary to our assumption.
Note that $v \in R$, otherwise $R \cup \{ba, av\}$ is a $\cp$-rainbow $F$-AAP.
Hence $R \cup \{ba, av\}$ is the desired odd cycle.
\end{proof}

\begin{claim}\label{claim: no five cycle}
There is no odd $\cp$-rainbow $F$-alternating cycle of length larger than $3$.
\end{claim}
\begin{proof}
Suppose  that $C$ is a $\cp$-rainbow $F$-alternating cycle with $|C|=2q+1$ for $q>1$. 
Contract  $C$ to one vertex, and remove from $\cp$ the set $\cp_C$ of the paths in $\cp$ represented by edges of $C$.
This results in a matching $F'$ of size $k':=k-q$ and a family $\cp'$ of $2k-q-1$ walks, each of which contains an $F'$-alternating path.
Since $q>1$, we have $2k-q-1 > 2k'$, so by the induction hypothesis on (1) there exists a $\cp'$-rainbow $F'$-AAP $R'$. 
Note that every vertex in $C$ is reachable from $v$ by an even (possibly empty) $\cp_C$-rainbow $F$-alternating path.
Therefore $R'$ can be extended by adding edges from $C$ to a $\cp$-rainbow $F$-AAP, a contradiction. 
\end{proof}

 Combining the above claims gives a rainbow $F$-alternating triangle $C$, namely $C=\{va, ab, bv\}$ for some vertex $v \not\in U$ and edge $ab \in F$.

As before, contract $C$ to a point $w$, and let $F'=F \setminus \{ab\}$. 
For every $P\in \cp$ define a path $P'$ as follows: 

\begin{itemize}
    \item If $P$ does not pass through $a,b$ or $v$, let $P'=P$.
    \item If  $P=P_1 v$ and it does not pass through $a$ or $b$, let $P' = P_1 w$.
   \item If $P=P_1abP_2$ or $P_1baP_2$, where $v$ possibly appears in $P_2$ but does not appear in $P_1$, let $P'=P_1w$.
\end{itemize}
Let $\cp'$ be the resulting family of $2k$ $F'$-AAPs.

\begin{claim}\label{claim:contract}
Let $X$ be a path containing $av$ and let $Y$ be a path containing $bv$.
Then 
$\cp'\setminus\{X',Y'\}$ is a $(k-1)$-badge.
\end{claim}
\begin{proof}
Suppose that there exists a $(\cp'\setminus\{X',Y'\})$-rainbow $F'$-AAP $R'$.
If $R'$ does not contain $w$, then it is also a $\cp$-rainbow $F$-AAP.
Otherwise, $R'$ arises from an odd length $\cp$-rainbow $F$-AAP $R$, whose final vertex is in $C$.
\begin{itemize}
    \item If $a$ is an endpoint of $R$, then $R \cup \{ab,bv\}$ is a $\cp$-rainbow $F$-AAP, since $bv \in Y$.
    \item If $b$ is an endpoint of $R$, then $R \cup \{ba,av\}$ is a $\cp$-rainbow $F$-AAP, since $av \in X$.
    \item If $R$ ends in $v$, then $R$ itself is a $\cp$-rainbow $F$-AAP.
\end{itemize}

So, we may assume that there is no such $R'$. 
The claim then follows from the induction hypothesis on (2).
\end{proof}

We know there exist two distinct paths $T_a$ that contains $av$ and $T_b$ that contains $bv$ in $\mathcal{P}$ because $C=\{va,ab,bv\}$ is a rainbow $F$-alternating triangle. 
Since there is no $\cp$-rainbow $F$-AAP, by Claim~\ref{claim:contract}, $\cp' \setminus \{T'_a,T'_b\}$ forms a $(k-1)$-badge, to be named $\cb'$.

\smallskip

{\bf Case I.} There is no origamistrip in $\cb'$ ending at $w$.

{\bf Case Ia. } One of $T'_a, ~T'_b$ is a path of length $>1$. This path then contains an edge not incident to $w$.
 Corollary~\ref{rmk:extraedge} yields then a $\cp'$-rainbow $F'$-AAP $R$.
Since in the presently considered case $\cb'$ has no edge incident to $w$, $R$  avoids $w$, so it is also a $\cp$-rainbow $F$-AAP.

{\bf Case Ib.}  $T'_a=wz$ and $T'_b=wt$ for some $z,t \not\in U$, meaning that $T_a=vabz$ and $T_b=vbat$.
If $z \neq t$, then $zbat$ is a $\cp$-rainbow $F$-AAP.
If $z = t$, then $\{T_a,T_b\}$ form a 1-origamistrip, which added to $\cb'$ makes $\cp$ a $k$-badge.

\smallskip

{\bf Case II.} There exists  an $m$-origamistrip $OS'$ in $\cb'$ containing $w$ as an endpoint.

Let $E(OS') \cap F = \{u_1v_1,\ldots,u_m v_m\}$ and let $z$ be the other endpoint of $OS'$. 
Then there are $2m$ paths of $\cp'\setminus \{T'_a, T'_b\}$ contained in $OS'$, say 
\begin{itemize}
    \item $Q'_1 = \ldots = Q'_m =P^A(OS')=z u_1v_1 u_2v_2 \ldots u_mv_m w $, and
    \item  $Q'_{m+1} = \cdots = Q'_{2m} = P^B(OS')=z v_1u_1 v_2u_2 \ldots v_mu_m w $.
\end{itemize}
Recall that each $Q_i'$ arose from a   path $Q_i \in \cp$.
For $i \le m$, $Q_i$  contains exactly one of $v_m a$, $v_m b$ and $v_m v$, and for $m+1\le i \le 2m$, $Q_i$  contains exactly one of $u_m a$, $u_m b$ and $u_m v$.

{\bf Case IIa.} $ab \not \in Q_i$ for all $i \le 2m$. 

In this case,  $Q_i'$ is equal to $Q_i$ for all $i$, with  $v$ replaced by $w$.
Then, replacing $w$ by $v$ in  $OS'$ results in an origamistrip $OS$ in $\cp$.
By the induction hypothesis, $\cp \setminus OS$ is a $(k-m)$-badge, hence $\cp$ is a $k$-badge.

{\bf Case IIb.} $ab \in Q_i$ for some $i$.  

Without loss of generality, $i \leq m$, and $a$ is closer than $b$ to $z$ on $Q_i$. So $Q_i$ is of the form $z u_1 v_1 \dots u_m v_m ab R$ for some path $R$.

To deal with this sub-case, the following will be used repeatedly.
\begin{claim}\label{claim:step1+2}
Let $r$ be  one element of  $\{a,b\}$, and $s$ the other element.    
\begin{enumerate}[(i)]
    \item If $Q_j$ contains $v_mr$ for some $j \leq m$, then 
    \[T_s = Q_j = zu_1 v_1 \ldots u_m v_m rsv.\]

\item If $Q_j$ contains $u_mr$ for some $j \geq m+1$, then \[T_s = Q_j = zv_1 u_1 \ldots v_m u_m rsv.\]
\end{enumerate}
\end{claim}
\begin{proof}
By symmetry, it suffices to show $(i)$.
Without loss of generality, we may assume $r = a$ and $s=b$.

We first claim that $Q_j$ ends with $bv$.
For otherwise,
$Q_j$ contains an edge $by$, where $y \not \in V(OS')$. 
If $y \not \in U$, then $vaby$ is a $\cp$-rainbow $F$-AAP. If $y \in U$, then it belongs to another origamistrip $OS$ in $\cb'$.
So it is possible to continue the path $vaby$ along one of the paths in $OS$ to get a $\cp$-rainbow $F$-AAP which ends at a vertex $v' \notin U$ with $v' \neq v$. To see the rainbow-ness, color $va$ by $T_a$, and  the edge  $by$ by $Q_j$.

Next we show $T_b = Q_j$.
By Claim~\ref{claim:contract}, $\left(\cb'\setminus \{Q'_j\}\right) \cup \{T'_b\} = \cp' \setminus \{T_a',Q'_j\}$ is also a $(k-1)$-badge, and since it shares all but one $F$-AAPs with $\cb'$, we must have  $T'_b=Q'_j$ (two badges cannot differ in one path, namely they cannot have symmetric difference of size $2$).
As the paths $T_b$ and $Q_j$ both begin with $vba$, they must be identical.
\end{proof}

\begin{claim}\label{claim:step3}
Let $\{r,s\}=\{a,b\}$.
\begin{enumerate}[(i)]
\item If $T_s = Q_j = zu_1 v_1 \ldots u_m v_m rsv$ for some $j \leq m$, then 
\[Q_{m+1} = \cdots = Q_{2m} = T_r = z v_1 u_1 \ldots v_m u_m srv.\]

\item If $T_s = Q_j = z v_1 u_1 \ldots v_m u_m rsv$ for some $j \geq m+1$, then 
\[Q_1 = \cdots = Q_m = T_r = z u_1 v_1 \ldots u_m v_m srv.\]
\end{enumerate}
\end{claim}
\begin{proof}
By  symmetry it suffices to prove $(i)$, and we assume without loss of generality $r = a$ and $s=b$.

Let $\ell \ge m+1$. 
We claim that $Q_\ell$ contains the edge $u_m b$.
Recall that $Q'_\ell$ ends with  $u_mw$.
So if $Q_\ell$ does not contain $u_mb$, then it must contain either $u_m v$ or $u_m a$.
\begin{itemize}
    \item If $Q_\ell$ contains $u_mv$, then the $5$-cycle 
    $\{u_mv, vb, ba, av_m, v_mu_m\}$ is a rainbow odd $F$-alternating cycle (where $vb$ represents $T_b$, $av_m$ represents $Q_j$ and $u_mv$ represents $Q_\ell$). 
    This contradicts Claim~\ref{claim: no five cycle}.
    
    \item If $Q_\ell$ contains the edge $u_ma$, then by Claim~\ref{claim:step1+2}, we obtain
    \[T_b=Q_\ell=zv_1u_1v_2u_2\ldots v_mu_mabv,\]
    but this contradicts $T_b=Q_j$. 
\end{itemize}
    Thus $Q_\ell$ must contain $u_m b$ for all $\ell \geq m+1$.
    Applying Claim~\ref{claim:step1+2} to each $\ell$, it follows \[T_a = Q_{m+1} = \cdots = Q_{2m} = zv_1u_1 \dots v_m u_m b a v.\]
\end{proof}

Using Claim~\ref{claim:step1+2} and Claim~\ref{claim:step3}, we determine $Q_j$ for every $j$.
\begin{claim}\label{claim:uncontract}
If either $Q_i$ contains $v_m a$ for some $i \leq m$ or $Q_i$ contains $u_m b$ for some $i \geq m+1$ then
\begin{enumerate}[(a)]
\item $Q_1 = \cdots = Q_m=T_b=  zu_1v_1 \dots u_mv_m a b v$, and
    \item $Q_{m+1} = \cdots = Q_{2m} = T_a=zv_1u_1 \dots v_m u_m b a v$.
\end{enumerate}
That is, they form an $(m+1)$-origamistrip $OS$.
\end{claim}
\begin{proof}
By symmetry, it suffices to show when $Q_i$ contains $v_m a$ for some $i \leq m$.
By Claim~\ref{claim:step1+2}, $Q_i = T_b = z u_1 v_1 \ldots u_m v_m a b v$.

Since $Q_i = T_b = z u_1 v_1 \ldots u_m v_m a b v$, Claim~\ref{claim:step3} gives  \[Q_{m+1} = \cdots = Q_{2m} = T_a = zv_1u_1 \dots v_m u_m b a v.\]
Applying Claim~\ref{claim:step3} with $Q_{m+1} = T_a = zv_1u_1 \dots v_m u_m b a v$, we obtain \[Q_1 = \cdots = Q_m = T_b = zu_1v_1 \dots u_mv_m a b v,\]
as required.
\end{proof}

This completes the proof of Theorem \ref{thm:rb.a.a.path}.\phantom\qedhere
\end{proof}

\subsection{The proof of Theorem~\ref{thm: main}}
In this section, we derive Theorem~\ref{thm: main} from Theorem \ref{thm:rb.a.a.path}.

Let $G$ be a graph.
Let $n \geq 3$ and $E_1,\ldots,E_{3n-3}$ matchings of size $n$ in $G$.
Since $3n-3 \geq 3(n-1)-2$, Theorem~\ref{3nminus2} guarantees the existence of a rainbow matching of size $n-1$, say $F$.
Without loss of generality, let $F = \{e_{1},\ldots, e_{n-1}\}$ where $e_i \in E_i$ for each $i \in [n-1]$.
Since $E_i$ is a matching of size $n$ for each $i \in [3n-3]\setminus[n-1]$, the set $E_i \cup F$ contains an $F$-AAP $P_i$.
If there exists a rainbow $F$-AAP $Q$, then $(Q \setminus F) \cup (F \setminus Q)$ is a rainbow matching of size $n$.

Thus, assuming negation, there is no rainbow $F$-AAP.
Then part (2) of Theorem~\ref{thm:rb.a.a.path} implies $P_n, \ldots, P_{3n-3}$ form an $(n-1)$-badge $\mathcal{B}$.
Note that, since $E_1,\ldots,E_{n-1}$ are themselves matchings of size $n$, each contains an edge that does not belong to $F$.

\medskip

\noindent {\bf Case I.} Suppose $\cb$ is a single $(n-1)$-origamistrip $OS$.
Let $x$ and $y$ be its endpoints.
We may assume that for each $i \leq n-1$, $e_i$ is  the $i$-th edge of $F$ on the path $P^A(OS)$ starting at $x$.
We will apply Lemma~\ref{lem: add an edge} with the edges of $E_1 \setminus F \neq \emptyset$ to find a rainbow matching of size $n$.
There are two sub-cases:
\begin{enumerate}
    \item Assume one of the endpoints of some edge $e \in E_1 \setminus F$ is not a vertex in $OS$.
    If $e$ is incident to either $x$ or $y$, then take a rainbow matching $R \subset A(OS)$ of size $n-1$ that is vertex-disjoint from $e$.
    Then $R \cup \{e\}$ is a rainbow matching of size $n$.
    Thus we may assume $e$ is not incident to none of $x$ and $y$.
    Now Lemma~\ref{lem: add an edge} gives a rainbow $F$-AAP, say $Q$. 
    Then $M:=(F \setminus Q) \cup (Q \setminus F)$ is a matching of size $n$. Either $M$ does not contain $e_1$ and is already rainbow, or will become rainbow upon replacing $e_1$ with one of the edges from $x$ to $e_1$ in $OS$ (as there are $2n-2$ different colors present among those two edges).
    \item Otherwise, since $OS$ has $2n$ vertices and $E_1$ is a matching of size $n$, $E_1$ must be a perfect matching on $V(OS)$. If $zy \in E_1$ with some $z \neq x$, then Lemma~\ref{lem: add an edge} gives a rainbow matching of size $n$ as above. If $xy \in E_1$, then $\{xy, e_A, e_B, e_3, \ldots, e_{n-1}\}$ is a rainbow matching of size $n$, where $e_A \in A(OS)$ and $e_B \in B(OS)$ are the edges of $OS$ that connect $e_1$ and $e_2$.
\end{enumerate}
    Note that the last step requires $n \geq 3$.
    Indeed, $(3n-3, n) \to n$ is not true for $n \leq 2$.
    This will be discussed at the end of Section~\ref{sec:cooperative}: see Example~\ref{ex:K4}.

\smallskip

\noindent {\bf Case II.} There are two  origamistrips in $\cb$.
Let $OS$ be the origamistrip that contains $e_1$.
Take $e \in E_1 \setminus F$ and let $\cb' = \cb \cup \{\{e\}\}$.
By Corollary~\ref{rmk:extraedge}, there exists a $\cb'$-rainbow $F$-AAP $Q$ containing $e$.
We claim that $(Q \setminus F) \cup (F \setminus Q)$ is a rainbow matching of size $n$.
This follows by observing that for each $n \leq i \leq 3n-3$, $E_i = (P_i \setminus F) \cup (F \setminus M(OS_i))$, where $OS_i$ is the origamistrip that contains $P_i$.

\bigskip

\section{A cooperative generalization}\label{sec:cooperative}

In Theorem \ref{thm: main} the matching of size $n$ was rainbow with respect to
a collection edge sets, each being itself a matching of size $n$. 
As with various other results on rainbow sets, there are also ``cooperative'' versions, in which the assumption is not on each set individually, but on the union of sub-collections. For example, 
B\'ar\'any's famous colorful Carath\'{e}odory theorem \cite{barany} was given in \cite{hpt} a cooperative version. There, the requirement on individual sets-that their convex hull contains the origin-is relaxed to the condition that the union of every two sets convexes the origin.  
In this section we prove two results in this vein:
\begin{enumerate}
    \item[(a)] A cooperative version of Theorem \ref{3nminus2} (Theorem \ref{cor: cooperating any matching}), and
    \item[(b)] A generalization of Theorem \ref{thm: main} (Theorem \ref{cor:realmain}), valid  for $n \ge 3$.
\end{enumerate}

As before, the core of the proofs will be in  results on rainbow paths in networks. Our point of departure  is:
 
\begin{theorem}\label{thm: cooperative any}
Let $F$ be a matching of size $k$ in a graph $G$, let $t$ be a non-negative integer, and let $\mathcal{A} = (A_1,\ldots,A_m)$ be a family of sets of edges, satisfying the condition that the union of any $t+1$ sets $A_i$ contains an $F$-AAP.
Let $J=\{j \mid A_j \subseteq F\}$. If there does not exist an $\mathcal{A}$-rainbow $F$-AAP 
then
\begin{enumerate}
\item $m=|\mathcal{A}| \leq 2k+t$, and
\item If $|\mathcal{A}| = 2k+t$, then $|J|=t$ and the $2k$ sets $A_j$, $j \notin J$ form a generalized $k$-badge.
\end{enumerate}
\end{theorem}
Recall ``generalized badges'' are badges with $F$-edges arbitrarily added and removed (see Definition \ref{genbadge}).
Note that Theorem~\ref{thm:rb.a.a.path} is the case $t = 0$ of Theorem~\ref{thm: cooperative any}.

\begin{proof}[Proof of Theorem~\ref{thm: cooperative any}]
By induction on $k + t$. Throughout the proof we assume, by negation, that there is no $\mathcal{A}$-rainbow $F$-AAP.
If $k = 0$, i.e. if $F  = \emptyset$, then  both parts of the theorem are obvious since any edge forms a rainbow $F$-AAP.

Consider next the case $t = 0$. 
By the condition of the theorem, $A_i$ contains an $F$-AAP $P_i$ for every $i \le m$. So (1) follows from Theorem~\ref{thm:rb.a.a.path}.
To prove (2), suppose $|\ca| = 2k$.
Since $P_i \nsubseteq F$, we have $|J|=0$, as required in (2).
Since $\{P_1,\ldots,P_{2k}\}$ do not have a rainbow $F$-AAP, they form a $k$-badge by Theorem~\ref{thm:rb.a.a.path}.
By Lemma~\ref{lem: add an edge}, $A_i \setminus P_i \subseteq F$  for every $i$, meaning that the sets $A_i$ form a generalized badge. 

So assume instead both $k,t > 0$.
It suffices to prove (2), since it implies (1). To see this, assume (2) and suppose $|\ca| > 2k+t$. 
By (2), we have $|J \cap [2k+t]| = t > 0$.
Let $i \in J \cap [2k+t]$.
Then again, (2) gives $|J \cap \left([2k+t+1]\setminus\{i\}\right)| = t$.
Thus $|J| \geq t+1$.
But $\bigcup_{j \in J}A_j \subset F$, so does not contain an $F$-AAP, contrary to assumption. 

If $|J| = t$, then, for each $i \notin J$, applying the condition of the theorem  to $J \cup \{i\}$ yields that  $A_i$ contains an $F$-AAP. Let $\ca' = \ca \setminus \{A_i: i \in J\}$. As in the above proof of the case $t=0$,  Theorem \ref{thm:rb.a.a.path}  and Corollary \ref{rmk:extraedge} imply $\ca'$ forms a generalized badge, as desired.

Now suppose $|J| < t$.
If $J \neq \emptyset$, say $A_i \subseteq F$ for some $i$, then the union of any $t$ sets in $\A' := \mathcal{A} \setminus \{A_i\}$ contains an $F$-AAP. 
Since $|\ca'| = 2k+t-1$ and $|J \setminus \{i\}| < t-1$, the induction hypothesis applied to $\A'$ gives a rainbow $F$-AAP.
Thus it is sufficient to show that if $A_i \setminus F \neq \emptyset$ for all $i$, then there is necessarily an $\ca$-rainbow $F$-AAP.
Since $\bigcup \mathcal{A} \cup F$ contains some $F$-AAP, we may assume that $A_1$ contains an edge $va$ where $v \notin \bigcup F$ and $ab \in F$. Let $F' = F \setminus \{ab\}$.

Suppose first that the edge $vb$ belongs to some $A_i$, ~$i \neq 1$,  say to $A_2$. 
Let $G'$ be the graph obtained from $G$ by contracting  $v,a,b$ to one vertex $w$.
Consider $A_i' = A_i\setminus\{va,vb,ab\}$ as  (possibly empty) edge sets in $G'$, where every edge of the form $xy$ for some $x \in \{a,b,v\}, y \not\in  \{a,b,v\}$ is replaced by $wy$.
If $\mathcal{A}'_1 = \{A_3',\ldots, A_{2k+t}'\}$ has a rainbow
$F'$-AAP in $G'$, then it can be extended to a rainbow $F$-AAP
in $G$, a contradiction.
So, we may assume that  $\mathcal{A}'_1$ has no rainbow $F'$-AAP.
Applying the induction hypothesis, and relabelling if necessary, we may assume that $A_{2k+1}',\ldots,A_{2k+t}' \subseteq F'$ and that $\cb := \{A_3',\ldots,A_{2k}'\}$ is a generalized $(k-1)$-badge.

Let $i \in \{2k+1, \ldots ,2k+t\}$. Since $A'_i \subseteq F'\subset F$,  recalling that  $A_i \setminus F \neq \emptyset$, it follows that $A_i$ contains $va$ or $vb$.
Without loss of generality, assume $va \in A_{2k+t}$. 
Let $\mathcal{A}'_2 = \{A_1',A_3',A_4',\ldots,A_{2k+t-1}'\}
=\mathcal{A}'_1 \cup \{A_1'\}\setminus \{A_{2k+t}'\}
$.
By the same reasoning as above, $\mathcal{A}'_2$ does not contain a rainbow $F'$-AAP.
Then necessarily $A_1' \setminus F' = \emptyset$, because $\cb \subseteq  \mathcal{A}'_2$ is a $(k-1)$-generalized badge, and the sets $A'_i \not\in \cb$ are, by definition, contained in $F'$.  
Therefore $A_1,A_{2k+1},\ldots,A_{2k+t} \subseteq F \cup \{va, vb\}$.
But then these $t+1$ sets have no $F$-AAP, a contradiction.

Thus  we may assume that $vb\not \in A_i$ for any $i>1$. 
Let $G_a = G - a$, $B_i = A_i \setminus \{a\}$ for each $i > 1$.

Let $\D = \{B_2,\ldots,B_{2k+t}\}$.
Then the union of any $t+1$ members of $\D$ contains an $F'$-AAP in $G_a$.
Since \[|\D| = 2k+t-1 > 2(k-1) + t  \;\;\;\text{and}\;\;\; |F'| = k-1,\]
by the induction hypothesis, there exists a $\D$-rainbow $F'$-AAP $P$ in $G_a$.
If the two endpoints of $P$ are $v$ and $b$, then $P \cup \{ab, va\}$ is an odd rainbow $F$-alternating cycle of length at least $5$ in $G$, and the argument proceeds as in the proof of Claim~\ref{claim: no five cycle}: contract the cycle to a vertex $w$, set aside the represented color sets, use the inductive hypothesis for (1) to find an AAP $R$ on what is left, and then use the edges of the odd cycle to make an $F$-AAP in $G$ if $R$ ends in $w$.
Thus we may assume that at least one of $v$ and $b$ is not an endpoint of $P$.
If $b$ is not an endpoint of $P$, then $P$ is also
an $\A$-rainbow $F$-AAP in $G$.
If $b$ is an endpoint of $P$ but $v$ is not, then $P \cup \{ab, va\}$ is an $\A$-rainbow $F$-AAP in $G$.
In all cases we have thus found the desired rainbow $F$-AAP for contradiction.
\end{proof}

We  now use  
 Theorem \ref{thm: cooperative any} to prove a cooperative rainbow matchings result. 
 We shall write  $(m, q, n)\to  k$ for the statement ``any $m$ sets of edges such that the union of every $q$ of them contains a matching of size $n$, have a rainbow matching of size $k$''. 
\begin{theorem}\label{cor: cooperating any matching}
If $n \geq 1$ and $t \geq 0$ then $(3n-2+t, t+1, n) \to n$. 
\end{theorem}
\begin{proof}
The proof goes along similar lines to Theorems \ref{drisko} and \ref{3nminus2}. We induct on $n$. The case $n=1$ is simple (a single edge is always rainbow), so we may assume $n > 1$.

Let $G$ be any graph and $E_1,\ldots,E_{3n-2+t}$  sets of edges in $G$ such that the union of every $t+1$ of them contains a matching of size $n$.
By the induction hypothesis, there exists a rainbow matching $F$ representing $E_j, ~j \in K$ for a set $K \subseteq  [3n-2+t]$ of size $n-1$.
Then $|[3n-2+t]\setminus K|=2(n-1)+t+1$.
For every set $I\subseteq  [3n-2+t] \setminus K$ of size  $t+1$ the set $\bigcup_{i \in I} (E_i \cup F)$ contains a matching of size $n$, and hence an $F$-AAP.
Hence, by Theorem~\ref{thm: cooperative any} (1), applied  with $k=n-1$, the family $(E_i \cup F \mid i \in  [3n-2+t]\setminus K)$  has a rainbow $F$-AAP $R$.
Since $R\setminus F$ and $F \setminus R$ are vertex disjoint matchings, $(R \setminus F) \cup (F \setminus R)$ is a rainbow matching of size $n$.
\end{proof}
Putting $t=0$ yields Theorem \ref{3nminus2} ($(3n-2,n)\to n$). The proof here is different from that in \cite{abchs}, in which the main weapons were the Edmonds-Gallai decomposition theorem and the Edmonds' blossom algorithm.

For $n \ge 3$ Theorem~\ref{cor: cooperating any matching} can be improved, to yield a cooperative version of Theorem \ref{thm: main}.

 \begin{theorem}\label{cor:realmain}
Let $t \geq 0$ and $n \ge 3$.
Then $(3n-3+t, t+1, n) \to n$.
\end{theorem}
\begin{proof}
Let $n \geq 3$, $t \geq 0$. 
Let $E_1,\ldots,E_{3n-3+t}$ be (possibly empty) edge sets in a graph $G$ such that the union of any $t+1$ of them contains a matching of size $n$.
By Theorem~\ref{cor: cooperating any matching}, we know
\[(3(n-1) - 2 + t, t+1, n-1) \to n-1.\]
Since $3n-3 + t \geq 3(n-1)-2+t$, we can find a rainbow matching $F$ of size $n-1$.
Without loss of generality, let $F = \{e_{1},\ldots, e_{n-1}\}$ where $e_i \in E_i$ for each $i \in [n-1]$.

Since $\bigcup_{i \in I} E_i$ contains a matching of size $n$ for each $I \subseteq [3n-3+t] \setminus [n-1]$ of size  $t+1$, the set $\left( \bigcup_{i \in I} E_i \right) \cup F$ of edges  contains an $F$-AAP.
If there exists a rainbow $F$-AAP $Q$, then $(Q \setminus F) \cup (F \setminus Q)$ is a rainbow matching of size $n$.

Thus, assuming negation, there is no rainbow $F$-AAP.
Then by part (2) of Theorem~\ref{thm: cooperative any}, we may assume that $E_{3n-2},\ldots,E_{3n-3+t} \subseteq F$ and $\cb = \{E_n,\ldots,E_{3n-3}\}$ forms an $(n-1)$-generalized badge.
For each $j \in [3n-3]$, let $D_j = E_j \cup E_{3n-2} \cup \cdots \cup E_{3n-3+t}$.
There are two cases: either there is only one origamistrip or there are more.

\smallskip

\noindent {\bf Case I.} Suppose there is a single $(n-1)$-origamistrip $OS$ in $\cb$.
Let $x$ and $y$ be its endpoints.
We may assume that for each $i \leq n-1$, $e_i$ is  the $i$-th edge of $F$ on the path $P^A(OS)$ starting at $x$.
Since $D_1$ contains a matching of size $n$, it must contain an edge not belonging to $F$.
We next proceed as in the proof of Lemma \ref{lem: add an edge} to use this edge to find a rainbow matching of size $n$.

\begin{claim}\label{claim: single OS-1}
If some $e \in D_1$ is not incident to any vertex of $OS$, then there is a rainbow matching of size $n$.
\end{claim}
\begin{proof}
Since there are $n-1$ copies of $P^A(OS)$ there is a rainbow matching $R \subseteq  A(OS) = P^A(OS) \setminus F$ of $\cb$ of size $n-1$.
Then $R \cup \{e\}$ is a rainbow matching of size $n$.
\end{proof}

\begin{claim}\label{claim: single OS-2}
If some $e \in D_1$ connects a vertex $z$ of $OS$ and a vertex not in $OS$, then there is a rainbow matching of size $n$.
\end{claim}
\begin{proof}
If $z \in \{x,y\}$, take a rainbow matching $R \subseteq A(OS)$ of $\cb$ of size $n-1$ that is vertex-disjoint from $e$.
Then $R \cup \{e\}$ is a rainbow matching of size $n$.
Thus we may assume $z \notin \{x,y\}$.

Recalling Observation \ref{qaqb} for the definition of $Q_x^z$, let $Q = Q_x^z + e$.
Then $Q \setminus F$ is an $\{E_1,E_n,E_{n+1},\ldots,E_{3n-3}\}$-rainbow matching of size $|Q \cap F| + 1$.
Since $Q$ contains $e_1$, $F \setminus Q$ is an $\{E_2,\ldots,E_{n-1}\}$-rainbow matching.
Thus $(Q \setminus F) \cup (F \setminus Q)$ is also rainbow (and its size is $n$, as needed).
\end{proof}

By Claim~\ref{claim: single OS-1} and Claim~\ref{claim: single OS-2}, we may assume that any edge of $D_1$ has both endpoints in $OS$.
Since $D_1$ contains a matching of size $n$ and $OS$ consists of $2n$ vertices, $D_1$ contains a perfect matching on $OS$.
In particular, it contains an edge $e$ that is incident to $y$.

Let $e = zy$. Suppose first $z \neq x$.
Let $Q' = Q_x^z + e$.
Then $Q' \setminus F$ is a rainbow matching of size $|Q' \cap F| + 1$.
Since $Q'$ contains $e_1$, it follows that $(Q' \setminus F) \cup (F \setminus Q')$ is a rainbow matching of size $n$.
So, we may assume that $e = xy$.
Since $n-1 \geq 2$, there are disjoint edges $e_A \in A(OS)$ and $e_B \in B(OS)$ that connect $e_1$ and $e_2$.
Then $\{e, e_A, e_B, e_3,\ldots, e_{n-1}\}$ is a rainbow matching of size $n$.

\smallskip

\noindent {\bf Case II.} Now suppose there are two different origamistrips in $\cb$.
For each $n \leq i \leq 3n-3$, let $OS_i$ be the origamistrip that contains $D_i \setminus F$.

\begin{claim}\label{claim: M_i}
Suppose $D_i \setminus F = A(OS_i)$.
Then $D_i \supseteq A(OS_i) \cup (F \setminus M(OS_i))$.
\end{claim}
\begin{proof}
Clearly, $A(OS_i) \subseteq D_i$.
Thus it is sufficient to show that $F \setminus M(OS_i) \subset D_i$.
By the assumption,  $OS_i$ is a $k$-origamistrip for some $k < n-1$, and hence $|D_i \setminus F| = k+1 < n$.
Since $D_i$ contains a matching of size $n$, there are edges of $D_i$ that are not incident to any vertex of $OS_i$.
By Corollary~\ref{rmk:extraedge}, all such edges are in $F \setminus OS_i$.
Since $|A(OS_i)| = k+1$ and $|F \setminus M(OS_i)| = n-1 - k$, it follows that $F \setminus M(OS_i) \subseteq  D_i$.
\end{proof}

Now let $OS$ be the origamistrip that contains $e_1$.
Since $D_1$ contains a matching of size $n$ and $|F| = n-1$, there exists an edge $e \in D_1 \setminus F$. 
In particular, since $D_1 \setminus F \subset E_1$, $e \in E_1$.
Let $\cb' = \cb \cup \{\{e\}\}$.
By Corollary~\ref{rmk:extraedge}, there exists a $\cb'$-rainbow $F$-AAP $Q$ containing $e$.
If $Q$ contains $e_1$, then $(Q \setminus F) \cup (F \setminus Q)$ is immediately a rainbow matching.
Otherwise, $(Q \setminus F) \cup (F \setminus Q)$ contains both $e$ and $e_1$.
We will show that $e_1$ is contained in some $E_{i_0}$ that did not participate in $Q$.
It follows $(Q \setminus F) \cup (F \setminus Q)$ is still rainbow, completing the proof.

If $e_1 \in E_{3n-2} \cup \cdots \cup E_{3n-3+t}$, then we can take $i_0 \in \{3n-2,\ldots 3n-3+t\}$ accordingly.
Suppose not.
Take any origamistrip $OS'$ that is different from $OS$.
Note that at least one of $Q \cap P^A(OS')$ and $Q \cap P^B(OS')$ should be empty.
Without loss of generality, we can assume $Q \cap P^A(OS') = \emptyset$.
Then take any $i_0 \in \{n,\ldots,3n-3\}$ with $E_{i_0} \setminus F \subset A(OS')$.
Then by Claim~\ref{claim: M_i}, $D_{i_0}$ contains $e_1$, implying that $e_1 \in E_{i_0}$, as required.
\end{proof}

The following examples show why the condition $n\ge 3$ is indeed necessary. 

\begin{example}\label{ex:K4}
Consider first the case $n = 1$. Then $t$ empty sets  vacuously satisfy the condition that any $t+1$ of them (satisfy any condition), and they do not have a rainbow matching of size $1$. 

For $n = 2$, let $G = ([4], \binom{[4]}{2}) = K_4$ and let $E_1 = \{12,34\}$, $E_2 = \{13,24\}$, $E_3 = \{14,23\}$, and $E_i = \emptyset$ for all $3 < i \leq t+3$.
Then the union of any $t+1$ of $E_i$'s contains a matching of size $2$, and yet there is no rainbow matching of size $2$.
\end{example}

In fact, these are the only examples showing that $3n-3+t$ colors do not suffice.
Let $t \geq 0$ and $F_1,\ldots,F_{t+3}$ be sets of edges of any graph $G$ such that the union of any $t+1$ $F_i$'s contains a matching of size $2$.
Suppose there is no rainbow matching of size $2$.
We may assume that $F_1 \neq \emptyset$.
Take any edge $e \in F_1$.
By Theorem~\ref{thm: cooperative any}, we may assume that $F_2 \cup \{e\}$ and $F_3 \cup \{e\}$ are $\{e\}$-AAPs from $u$ to $v$ forming a $1$-badge, and $F_i \subseteq  \{e\}$ for each $4 \leq i \leq t+3$.

Note that $F_1 \cup F_4 \cup F_5 \cup \ldots \cup F_{t+3}$ contains a matching of size $2$.
Since it is not rainbow, this implies the existence of an edge $f \neq e$ in $F_1$.
Moreover, $F_4 = \cdots = F_{t+3} = \emptyset$.
If $f \neq uv$, then there is a rainbow matching of size $2$ by Claim~\ref{claim: single OS-1} and Claim~\ref{claim: single OS-2}.
Thus $f = uv$, meaning that $F_1 = \{e, f\}$.
Then $F_2$ and $F_3$ do not contain $e$: otherwise, $\{e, f\}$ is a rainbow matching of size $2$.
This completes the proof.

\begin{remark}
A result in \cite{lee} implies a slightly weaker version of Theorem~\ref{cor:realmain}.
They gave a topological proof of $(3n-3+t, t+1, n) \to n$ for all $n \geq 1$ and $t > 0$ when at least $3n-2$ sets are nonempty.
\end{remark}

\section*{Acknowledgement}
We acknowledge the financial support from the Ministry of Educational and Science of the Russian Federation in the framework of MegaGrant no. 075-15-2019-1926 when the first author worked on Sections 1 through 3 of the paper.

The research of R.~Aharoni was supported in part by the United States--Israel Binational Science Foundation (BSF) grant no.\ 2006099, the Israel Science Foundation (ISF) grant no.\ 2023464 and the Discount Bank Chair at the Technion. This paper is part of a project that has received funding from the European Union's Horizon 2020 research and innovation programme under the Marie Sk\l{}dowska-Curie grant agreement no.\ 823748.

The research of J.~Briggs was supported by ISF grant no.\ 326/16, ISF grant no.\ 1162/15, and ISF grant no.\ 409/16
The research of J.~Kim was supported by BSF grant no.\ 2016077, ISF grant no.\ 1357/16.
The research of M.~Kim was supported by ISF grant no.\ 936/16.
This work was supported by the Institute for Basic Science (IBS-R029-C1).
This research was done while J.~Kim and M.~Kim were post-doctoral fellows at the Technion.

The authors thank the referees for their comments that helped us improve the readability of the paper.

\bibliographystyle{alpha}
\bibliography{3n-3references}

\newcommand{\etalchar}[1]{$^{#1}$}
\begin{thebibliography}{ABC{\etalchar{+}}19}

\bibitem[AB09]{ab}
R.~Aharoni and E.~Berger.
\newblock Rainbow matchings in {$r$}-partite {$r$}-graphs.
\newblock {\em Electron. J. Combin.}, 16(1):Research Paper 119, 9, 2009.

\bibitem[ABC{\etalchar{+}}19]{abchs}
R.~Aharoni, E.~Berger, M.~Chudnovsky, D.~Howard, and P.~Seymour.
\newblock Large rainbow matchings in general graphs.
\newblock {\em European J. Combin.}, 79:222--227, 2019.

\bibitem[ABKZ18]{abkz}
R.~Aharoni, E.~Berger, D.~Kotlar, and R.~Ziv.
\newblock Degree conditions for matchability in 3-partite hypergraphs.
\newblock {\em J. Graph Theory}, 87(1):61--71, 2018.

\bibitem[AHJ19]{ahj}
R.~Aharoni, R.~Holzman, and Z.~Jiang.
\newblock Rainbow fractional matchings.
\newblock {\em Combinatorica}, 39:1191--1202, 2019.

\bibitem[AKZ18]{akz}
R.~Aharoni, D.~Kotlar, and R.~Ziv.
\newblock Uniqueness of the extreme cases in theorems of {D}risko and
  {E}rd{\H{o}}s--{G}inzburg--{Z}iv.
\newblock {\em European J. Combin.}, 67:222--229, 2018.

\bibitem[B{\'{a}}r82]{barany}
I.~B{\'{a}}r\'{a}ny.
\newblock A generalization of carath\'{e}odory's theorem.
\newblock {\em Discrete Math.}, 40(2-3):141--152, 1982.

\bibitem[BGS17]{barat}
J.~Bar{\'a}t, A.~Gy{\'a}rf{\'a}s, and G.~N. S{\'a}rk{\"o}zy.
\newblock Rainbow matchings in bipartite multigraphs.
\newblock {\em Period. Math. Hungar.}, 74(1):108--111, 2017.

\bibitem[BR91]{br}
R.~A. Brualdi and H.~J. Ryser.
\newblock {\em Combinatorial Matrix Theory}.
\newblock Encyclopedia of Mathematics and its Applications. Cambridge
  University Press, 1991.

\bibitem[Dri98]{drisko}
A.~A. Drisko.
\newblock Transversals in row-latin rectangles.
\newblock {\em J. Combin. Theory Ser. A}, 84(2):181--195, 1998.

\bibitem[HL20]{lee}
A.~Holmsen and S.~Lee.
\newblock Leray numbers of complexes of graphs with bounded matching number.
\newblock {\em arXiv:2003.11270}, 2020.

\bibitem[HPT08]{hpt}
A.~F. Holmsen, J.~Pach, and H.~Tverberg.
\newblock Points surrounding the origin.
\newblock {\em Combinatorica}, 28(6):633--644, 2008.

\end{thebibliography}

\end{document}